\documentclass[12pt,reqno]{amsart}
\usepackage[margin=1in]{geometry}
\usepackage{amsmath, dsfont}

\DeclareMathOperator{\sgn}{sgn}

\DeclareMathOperator{\diag}{diag}

\newcommand{\Z}{\mathbb{Z}}
\newcommand{\Q}{\mathbb{Q}}
\newcommand{\R}{\mathbb{R}}
\newcommand{\C}{\mathbb{C}}

\newcommand{\oo}{\omega}
\newcommand{\p}{\partial}

\newtheorem{thm}{Theorem}[section]

\newtheorem{prop}[thm]{Proposition}
\newtheorem{lem}[thm]{Lemma}

\newtheorem{cor}[thm]{Corollary}

\theoremstyle{definition}

\newtheorem{definition}[thm]{Definition}
\newtheorem{example}[thm]{Example}
\newtheorem{rem}[thm]{Remark}

\newtheorem{set-up}[thm]{Set-up}

\usepackage{hyperref}
\usepackage{graphicx}
\usepackage{amssymb}
\usepackage{enumerate}
\usepackage{tikz}
\usepackage{todonotes}
\usepackage{framed}

\usepackage{color}
\usepackage{array}

\hypersetup{colorlinks}

\definecolor{indigo}{RGB}{51,0,102}
\definecolor{brightpurple}{RGB}{102,0,153}
\definecolor{fuchsia}{RGB}{180,51,180}
\definecolor{jolightpurple}{RGB}{188,171,240}

\hypersetup{colorlinks,
linkcolor=brightpurple,
filecolor=brightpurple,
urlcolor=indigo,
citecolor=fuchsia}
\normalbaselines
\numberwithin{equation}{section}

\usepackage{mathtools}

\input xy
\xyoption{all}
\title{Reeb Dynamics of the Link of the $A_n$ Singularity}
\author[Abbrescia]{Leonardo Abbrescia}
\address[Abbrescia]{Department of Mathematics, Michigan State University}
\curraddr{ 1016 Chester Rd. APT B14 \\
Lansing, MI 48912}
\email{leonardo@math.msu.edu}
\author[Huq-Kuruvilla]{Irit Huq-Kuruvilla}
\address[Huq-Kuruvilla]{Department of Mathematics, Columbia University}
\curraddr{  6580 Lerner Hall, 2920 Broadway  \\ New York, NY 10027}
\email{ih2271@columbia.edu}
\author[Nelson]{Jo Nelson}
\address[Nelson]{Institute for Advanced Study and Columbia University}
\curraddr{Room 509, MC 4403, 2990 Broadway \\ New York, NY 10027}
\email{nelson@math.columbia.edu}
\author[Sultani]{Nawaz Sultani}
\address[Sultani]{Department of Mathematics, Columbia University}
\curraddr{ 6580 Lerner Hall, 2920 Broadway\\ New York, NY 10027}
\email{njs2155@columbia.edu}

\begin{document}
\maketitle
\begin{abstract}
The link of the $A_n$ singularity, $L_{A_n} \subset \C^3$ admits a natural contact structure $\xi_0$ coming from the set of complex tangencies.  The canonical contact form $\alpha_0$ associated to $\xi_0$ is degenerate and thus has no isolated Reeb orbits.  We show that there is a nondegenerate contact form for a contact structure equivalent to $\xi_0$ that has two isolated simple periodic Reeb orbits.  We compute the Conley-Zehnder index of these simple orbits and their iterates.  From these calculations we compute the positive $S^1$-equivariant symplectic homology groups for $\left(L_{A_n}, \xi_0 \right)$. In addition, we prove that $\left(L_{A_n}, \xi_0 \right)$ is contactomorphic to the Lens space $L(n+1,n)$, equipped with its canonical contact structure $\xi_{std}$.


\end{abstract}

\setcounter{tocdepth}{2}
\tableofcontents

\section{Introduction and Main results }

The classical topological theory of isolated critical points of complex polynomials relates the topology of the link of the singularity to the algebraic properties of the singularity \cite{M}.  More generally, the link of an irreducible affine variety $A^n \subset \C^N$  with an isolated singularity at $\mathbf{0}$ is defined by $L_A = A \cap S_\delta^{2N+1}$.  For sufficiently small $\delta$, the link $L_A$ is a manifold of real dimension $2n-1$, which is an invariant of the germ of $A$ at $\mathbf{0}.$  The links of Brieskorn varieties can sometimes be homeomorphic but not always diffeomorphic to spheres \cite{Br}, a preliminary result which further motivated the study of such objects.    Recent developments in symplectic and contact geometry have shown that the algebraic properties of a singularity are strongly connected to the contact topology of the link and symplectic topology of (the resolution of) the variety.  A wide range of results demonstrating the power of investigating the symplectic and contact perspective of singularities include \cite{K}, \cite{O}, \cite{McL}, \cite{R}, \cite{Se}, \cite{U}.

 In this paper we study the contact topology of the link of the $A_n$ singularity, providing a computation of positive $S^1$-equivariant symplectic homology.  This is done via our construction of an explicit nondegenerate contact form and the computation of the Conley-Zehnder indices of the associated simple Reeb orbits and their iterates.  Our computations show that positive $S^1$-equivariant symplectic homology is a free $\Q[u]$ module of rank equal to the number of conjugacy classes of the finite subgroup $A_n$ of SL$(2;\C)$.  This provides a concrete example of the relationship between the cohomological McKay correspondence and symplectic homology, which is work in progress by McLean and Ritter \cite{MR}.   As a result, the topological nature of the singularity is reflected by qualitative aspects of the Reeb dynamics associated to the link of the $A_n$ singularity.
 
 The link of the $A_n$ singularity is defined by
\begin{equation}\label{linkeq}
L_{A_n} =f_{A_n}^{-1}(0) \cap S^5 \subset \C^3, \ \ \ f_{A_n}=z_0^{n+1} + 2 z_1 z_2.
 \end{equation}
 It admits a natural contact structure coming from the set of complex tangencies, 
  \[
  \xi_0:=TL_{A_n} \cap J_0(TL_{A_n}).
  \]
The contact structure can be expressed as the kernel of the canonically defined contact form, 
\[
\alpha_0 = \frac{i}{2} \left( \sum_{j=0}^m ( z_j d\bar{z}_j -\bar{z}_jdz_j ) \right)\bigg \vert_{L_{A_n}}.
\]
The contact form $\alpha_0$  is degenerate and hence not appropriate for computing Floer theoretic invariants as the periodic orbits of the Reeb vector field defined by
\[
{\alpha_0}(R_{\alpha_0})=1, \ \ \ \iota_{R_{\alpha_0}}d\alpha_0 =0,
\]
are not isolated.

Our first result is to construct a nondegenerate contact form $\alpha_\epsilon$ such that $( L_{A_n}, \ker \alpha_0)$ and $( L_{A_n}, \ker \alpha_\epsilon)$ are contactomorphic.  Define the Hamiltonian on $\C^3$ by 
\[
\begin{array}{rlcl}
H:& \C^3 &\to & \R \\
&  (z_0,z_1,z_2) &\mapsto & |z|^2 + \epsilon(|z_1|^2 - |z_2|^2), \\
\end{array}
\]
 where $\epsilon$ is chosen so that $H > 0$ on $S^5$.  As will be shown, 
\begin{equation}\label{alphaepsilon}
 \alpha_\epsilon = \frac{1}{H} \left[ \frac{(n+1)i}{8}\left(z_0d\overline{z}_0 - \overline{z}_0 dz_0\right) + \frac{i}{4} \left(z_1 d\bar{z}_1 - \bar{z}_1 dz_1 + z_2 d\bar{z}_2 - \bar{z}_2 dz_2\right)\right], 
 \end{equation}
is a nondegenerate contact form.  We also find the simple Reeb orbits of $R_{\alpha_\epsilon}$ and compute the associated Conley-Zehnder index with respect to the canonical trivialization of $\C^3$ of their iterates.

\begin{thm}\label{CZcomputation}
The 1-form $\alpha_\epsilon$ is a nondegenerate contact form for $L_{A_n}$ such that $(L_{A_n}, \ker \alpha_0)$ and $( L_{A_n}, \ker \alpha_\epsilon)$ are contactomorphic.  The Reeb orbits of $R_{\alpha_\epsilon}$ are defined by
\begin{align*}
\gamma_+(t) & = (0,e^{2i(1 + \epsilon)t},0) \quad \quad 0 \le t \le \frac{\pi}{1 + \epsilon}\\
\gamma_-(t) & = (0,0,e^{2i(1 - \epsilon)t}) \quad  \quad 0 \le t \le \frac{\pi}{1 - \epsilon}.
\end{align*}
The Conley-Zehnder index for $\gamma = \gamma_{\pm}^N$ in $0 \le t \le \frac{N\pi}{1 \pm \epsilon}$ is
\begin{align}\label{CZeq}
\mu_{CZ}(\gamma_{\pm}^N) = 2\left( \left\lfloor \frac{2N}{(n+1)(1 \pm \epsilon)}\right \rfloor + \left\lfloor \frac{N(1 \mp \epsilon)}{1 \pm \epsilon} \right\rfloor - \left \lfloor \frac{2N}{1 \pm \epsilon} \right \rfloor \right) + 2N + 1.
\end{align}
\end{thm}

\begin{rem}
If $\epsilon$ is chosen such that $0 <\epsilon \ll \frac{1}{N}$ then (\ref{CZeq}) can be further simplified:
\begin{equation}
\begin{array}{lcl}
\mu_{CZ}(\gamma_{-}^N)& = &2 \left\lfloor \dfrac{2N}{(n+1)(1 - \epsilon)}\right \rfloor + 1;\\ 
&&\\
\mu_{CZ}(\gamma_{+}^N)& = &2 \left\lfloor \dfrac{2N}{(n+1)(1 + \epsilon)}\right \rfloor  + 1. \\

\end{array}
\end{equation}
\end{rem}

The proof of Theorem \ref{CZcomputation} is obtained by adapting methods of Ustilovsky  \cite{U} to obtain both $\alpha_\epsilon$ and to compute the Conley-Zehnder indices.    The Conley-Zehnder index is a Maslov index for arcs of symplectic matrices and defined in Section \ref{CZsection}.  These paths of matrices are obtained by linearizing the flow of the Reeb vector field along the Reeb orbit and restricting to $\xi_0$.  To better understand the spread of the Reeb orbits and their iterates in various indices, we have the following example.

\begin{example}
Let $n=2$ and $0 <\epsilon \ll \frac{1}{10}$.  
\[
\begin{array}{lc c clcc}
\mu_{CZ}(\gamma_-  ) &=&1,&\ \ \ \ \ \ &\mu_{CZ} (\gamma_+ )&=&1 \\
\mu_{CZ}(\gamma_-^2 )& =&3,&\ \ \ \ \ \ &\mu_{CZ}(\gamma_+^2)& =&3 \\
\mu_{CZ}(\gamma_-^3 )&=&5, &\ \ \ \ \ \ &\mu_{CZ}(\gamma_+^3 )&=&3  \\
\mu_{CZ}(\gamma_-^4 )&=& 5,&\ \ \ \ \ \ &\mu_{CZ}(\gamma_+^4 )& =&5\\
\mu_{CZ}(\gamma_-^5 )& =& 7,&\ \ \ \ \ \ &\mu_{CZ}(\gamma_+^5 )& =&7\\
\mu_{CZ}(\gamma_-^6 )&=& 9,&\ \ \ \ \ \ &\mu_{CZ}(\gamma_+^6) &=& 7\\
\mu_{CZ}(\gamma_-^7 )&=& 9,&\ \ \ \ \ \ &\mu_{CZ}(\gamma_+^7) &=& 9 \\
\end{array}
\]
It is interesting to note that spread of integers is not uniform between $\mu_{CZ}(\gamma_-^N)$  and $\mu_{CZ}(\gamma_+^N),$ and where these jumps in index occur.  However, we see that there are $n=2$ Reeb orbits with Conley Zehnder index 1 and $n+1=3$ orbits with Conley Zehnder index $2k+1$ for each $k\geq1$.  
\end{example}

\begin{rem}\label{freehtpy}
Extrapolating this to all values of $n$ and $N$ demonstrates that the numerology of the Conley-Zehnder index realizes the number of free homotopy classes of $L_{A_n}$.  Recall that  $[\Sigma L_{A_n}] = \pi_0(\Sigma L_{A_n}) = \pi_1(L_{A_n})/\{\mbox{conjugacy classes}\}$ and $H_1(L_{A_n}, \Z) = \Z_{n+1}$.  The information that the $n+1$-th iterate of $\gamma_\pm$ is the first contractible Reeb orbit is also encoded in the above formulas.  Qualitative aspects of the Reeb dynamics reflect this topological information in the following computation of a Floer-theoretic invariant of the contact structure $\xi_0$.
\end{rem}

Theorem \ref{CZcomputation} allows us to easily compute  positive $S^1$-equivariant symplectic homology $SH_*^{+,S^1}$.  Symplectic homology is a Floer type invariant of symplectic manifolds, with contact type boundary, see also \cite{biased}.    Under additional assumptions, one can prove that the positive $S^1$-equivariant symplectic homology $SH_*^{+,S^1}$ is in fact an invariant of the contact structure;  see \cite[Theorems 1.2 and 1.3]{GuSH} and \cite[Section 4.1.2]{BO}.  Because of the behavior of the Conley-Zehnder index in Theorem \ref{CZcomputation}, we can directly compute $SH_*^{+,S^1}(L_{A_n}, \xi_0)$ and conclude that it is a contact invariant.  As a result, the underlying topology of the manifold determines qualitative aspects of any Reeb vector field associated to a contact form defining $\xi_0$.   

\begin{thm}\label{linksh}
The positive $S^1$-equivariant symplectic homology of $(L_{A_n}, \xi_0)$ is 
\[
SH^{+,S^1}_*(L_{A_n}, \xi_0) = \left\{  \begin{array}{cl}
    \Q^n & *  =1   \\
    \Q^{n+1} & * \geq 3, \mbox{ odd } \\
       0 & * \ \mbox{ else } \\
\end{array} \right.
\]
\end{thm}

\begin{proof}
To obtain a contact invariant from $SH^{+,S^1}_*$ we need to show in dimension three that all contractible Reeb orbits $\gamma$ satisfy  $\mu_{CZ}(\gamma)\geq3$;  see \cite[Theorems 1.2 and 1.3]{GuSH} and \cite[Section 4.1.2]{BO}.  The first iterate of $\gamma_\pm$ which is contractible is  the $(n+1)$-th iterate, and by Theorem \ref{CZcomputation}, will always satisfy $\mu_{CZ}(\gamma_\pm)\geq3$.

If $\alpha$ is a nondegenerate contact form such that the Conley-Zehnder indices of all periodic Reeb orbits are lacunary, meaning they contain no two consecutive numbers, then we can appeal to \cite[Theorem 1.1]{GuSH}. This result of Gutt allows us to conclude that over $\Q$-coefficients the differential for $SH^{S^1,+}$ vanishes.  In light of Theorem \ref{CZcomputation} we obtain the above result.
\end{proof}

 Remark \ref{freehtpy} yields the following corollary of Theorem \ref{linksh}, indicating a Floer theoretic interpretation of the McKay correspondence \cite{IM} via the Reeb dynamics of the link of the $A_n$ singularity.  The $A_n$ singularity is the singularity of $f^{-1}_{A_n}(0)$, where $f_{A_n}$ is described as (\ref{linkeq}).  This is equivalent to its characterization as the absolutely isolated double point quotient singularity of $\C^2/A_n$, where $A_n$ is the cyclic subgroup of SL$(2;\C)$; see Section \ref{contactgeomlens}.  The cyclic group $A_n$ acts on $\C^2$ by $(u,v) \mapsto \left(e^\frac{2\pi i}{n+1}u, e^\frac{2\pi in}{n+1}v\right)$. 
 \begin{cor}
The positive $S^1$-equivariant symplectic homology  $SH^{+,S^1}_*(L_{A_n}, \xi_0)$  is a free $\Q[u]$ module of rank equal to the number of conjugacy classes of the finite subgroup $A_n$ of $\mbox{\em SL}(2;\C)$.
\end{cor}

\begin{rem}
Ongoing work of Nelson \cite{jocompute} and Hutchings and Nelson \cite{HN3} is needed in order to work under the assumption that a related Floer-theoretic invariant, cylindrical contact homology, is a well-defined contact invariant of $(L_{A_n},\xi_{0})$.  Once this is complete, the index calculations provided in Theorem \ref{CZcomputation} show that  positive $S^1$-equivariant symplectic homology and cylindrical contact homology agree up to a degree shift.   

In \cite{BO} Bourgeois and Oancea prove that there are restricted classes of contact manifolds for which once can prove that cylindrical contact homology (with a degree shift) is isomorphic to the positive part of $S^1$-equivariant symplectic homology, when both are defined over $\Q$-coefficients.  Their isomorphism relies on having transversality for a generic choice of $J,$ which is presently the case  for unit cotangent bundles $DT^*L$ such that dim $L \geq 5$ or when $L$ is Riemannian manifold which admits no contractible closed geodesics \cite{BOcorrig}. Our computations confirm that their results should hold for many more closed contact manifolds.

\end{rem}

 Our final result is an explicit proof that $(L_{A_n}, \xi_0)$ and the lens space $(L(n+1,n), \xi_{std})$ are contactomorphic.  The lens space 
 \[ L(n+1,n) = S^3/\big((u,v) \sim (e^{2\pi i/(n+1)}u,e^{2\pi ni/(n+1)}v)\big) 
  \]
admits a contact structure, which is induced by the one on $S^3$ and can be expressed as the kernel of the following contact form, 
\[
\lambda_{std}= \frac{i}{2}  ( u d\bar{u} -\bar{u}du +v d\bar{v} -\bar{v}dv).
\]

\begin{thm}\label{lenslinkcontacto}
The link of the $A_n$ singularity $(L_{A_n}, \xi_0=\ker \alpha_0)$ and the lens space $(L(n+1,n), \xi_{std}=\ker \lambda_{std})$ are contactomorphic.
\end{thm}

Theorems \ref{linksh} and \ref{lenslinkcontacto} allow us to reprove the following result of van Koert and Kwon  \cite{O}.  Since $(L_{A_n}, \xi_0)$ and   $(L(n+1,n), \xi_{std})$  are contactomorphic and $SH_*^{S^1,+}$ is a contact invariant,     $SH_*^{S^1,+}(L(n+1,n),\xi_{std}) =SH_*^{S^1,+}(L_{A_n}, \xi_0)$.  

\begin{thm}[Appendix A \cite{O}]
The positive $S^1$-equivariant symplectic homology of $(L(n+1,n),\xi_{std})$ is
\[
SH^{+,S^1}_*(L(n+1,n), \xi_{std}) = \left\{  \begin{array}{cl}
    \Q^n & *  =1   \\
    \Q^{n+1} & * \geq 3, \mbox{ odd } \\
       0 & * \ \mbox{ else } \\
\end{array} \right.
\]
\end{thm}

Their proof relies on the following nondegenerate contact form on $(L(n+1,n),\xi_{std})$. If $a_1,a_2$ are any rationally independent positive real numbers then
\[ \lambda_{a_1,a_2} = \frac{i}{2} \sum_{ j = 1}^2 a_j(z_j d\overline{z}_j - \overline{z}_j dz_j)\]
 is a nondegenerate contact form for $(L(n+1,n), \xi_{std})$. The simple Reeb orbits on  $L(n+1,n)$ are given by
\begin{align*}
\gamma_1 & = (e^{it/a_1},0) \quad \quad 0 \le t \le \frac{ 2 a_1\pi}{n+1}, \\
\gamma_2 & = (0,e^{it/a_2}) \quad \quad 0 \le t \le \frac{2a_2\pi}{n+1},
\end{align*}
which descend from the simple isolated Reeb orbits on $S^3$.  Again, the $n+1$ different free homotopy classes associated to this lens space are realized by covers of the isolated Reeb orbits $\gamma_i$ for $i=1$ or $2$. The Conley-Zehnder index for $\gamma_1^N$ is 
\begin{equation}\label{CZlens} \mu_{CZ}(\gamma_1^N) = 2\left(\left\lfloor \frac{N}{n+1}\right\rfloor + \left\lfloor \frac{N a_1}{(n+1)a_2}\right\rfloor\right) + 1,
\end{equation}
with a similar formula holding for $\gamma_2^N$.


   \bigskip

\textbf{Outline} The necessary background is given in Section \ref{background}.  The construction of a nondegenerate contact form and the proof of Theorem \ref{CZcomputation} is given in Section \ref{CZcomputationsection}.  The proof of Theorem \ref{lenslinkcontacto} is given in Section \ref{sectionlinklens}.

\section*{Acknowledgements}
Abbrescia, Huq-Kuruvilla, and Sultani thank Dr. Jo Nelson and Robert Castellano for their patient guidance and tutelage for this project. They have expanded our mathematical knowledge and have set us on the right path for further research. We especially thank Dr. Nelson for going above and beyond the call of an REU instructor.   We are very grateful to the referee for their numerous helpful suggestions on improving the exposition of this paper.  

We thank Dr. Robert Lipschitz and the Columbia University math REU program for giving us the opportunity to bring this project to fruition. The REU program was partially funded by NSF Grant DMS-0739392. Since graduating from Columbia, Leonardo Abbrescia is supported in part by a NSF graduate research fellowship.  Jo Nelson is supported by NSF grant DMS-1303903, the Bell Companies Fellowship, the Charles Simonyi Endowment, and the Fund for Mathematics at the Institute for Advanced Study.

\section{Background}\label{background}
In these sections we recall all the necessary symplectic and contact background which is needed to prove Theorems \ref{CZcomputation} and \ref{lenslinkcontacto}.
\subsection{Contact Structures} \hspace{\fill} \\
First we recall some notions from contact geometry.  
\begin{definition}
Let $M$ be a manifold of dimension $2n+1$.  A \textbf{contact structure} is a maximally non-integrable hyperplane field $\xi=\mbox{ker }\alpha \subset TM$.  
\end{definition}

\begin{rem}
The kernel of a 1-form $\alpha$ on $M^{2n+1},$ $\xi=\ker \alpha$, is a contact structure whenever 
\[
\alpha \wedge (d\alpha)^n \neq 0, 
\]
which is equivalent to the condition that $d\alpha$ be nondegenerate on $\xi$.
\end{rem}
Note that the contact structure is unaffected when we multiply the contact form $\alpha$ by any positive or negative function on $M$.  We say that two contact structures $\xi_0=\mbox{ker } \alpha_0$ and $\xi_1=\mbox{ker }\alpha_1$ on a manifold $M$ are \textbf{contactomorphic} whenever there is a diffeomorphism $\psi:M \to M$ such that $\psi$ sends   $\xi_0$ to $\xi_1$:
\[
\psi_*(\xi_0)=\xi_1
\]
If a diffeomorphism $\psi: M\to M$ is in fact a contactomorphism then there exists a non-zero function $g:M \to \R $ such that $\psi^*\alpha_1=g\alpha_0$. Finding an explicit contactomorphism often proves to be a rather difficult and messy task, but an application of Moser's argument yields Gray's stability theorem, which essentially states that there are no non-trivial deformations of contact structures on a fixed closed manifold.  

First we give the statement of Moser's Theorem, which says that one cannot vary a symplectic structure by perturbing it within its cohomology class.  Recall that a \textbf{symplectic structure} on a smooth manifold $W^{2n}$ is a nondegenerate closed 2-form $\omega \in \Omega^2(W)$.

\begin{thm}[Moser's theorem] \cite[Thm 3.17]{MD} \label{moser}
Let $W$ be a closed manifold and suppose that $\oo_t$ is a smooth family of cohomologous symplectic forms on $W$.  Then there is a family of diffeomorphisms $\Psi_t$ of $W$ such that
\[
\Psi_0=\mbox{id},  \ \ \ \psi^*_t\oo_t=\oo_0.
\]
\end{thm}
 
The aforementioned contact analogue of Moser's theorem is Gray's stability theorem, stated formally below.

\begin{thm}[Gray's stability theorem] \cite[Thm 2.2.2]{G}
Let $\xi_t, \ t \in [0,1]$, be a smooth family of contact structures on a closed manifold $V$.  Then there is an isotopy $(\psi_t)_{t\in [0,1]}$ of $V$ such that 
\[ {\psi_t}_*(\xi_0) = \xi_t \  \mbox{ for each } t \in [0,1]    \]
\end{thm}


Next we give the most basic example of a contact structure.

\begin{example}
\em
Consider $\R^{2n+1}$ with coordinates $(x_1, y_1,...,x_n,y_n,z)$ and the 1-form
\[
\alpha=dz+\sum_{j=1} ^n x_jdy_j.
\]
Then $\alpha$ is a contact form for $\R^{2n+1}$.  The contact structure $\xi=\mbox{ker }\alpha$ is called the standard contact structure on $\R^{2n+1}$
\end{example}
As in symplectic geometry, a variant of Darboux's theorem holds.  This states that locally all contact structures are diffeomorphic to the standard contact structure on $\R^{2n+1}$.

A contact form gives rise to a unique Hamiltonian-like vector fields as follows.  

\begin{definition}
 For any contact manifold $(M, \xi=\mbox{ker }\alpha)$  the \textbf{Reeb vector field} $R_\alpha$ is defined to be the unique vector field determined by $\alpha$: 
\[
\iota(R_\alpha)d\alpha=0, \ \ \ \alpha(R_\alpha)=1.   
\]
We define the Reeb flow of $R_\alpha$ by $\varphi_t: M \to M$, $\dot{\varphi_t} = R_\alpha(\varphi_t)$.   
\end{definition}

The first condition says that $R_\alpha$ points along the unique null direction of the form $d\alpha$ and the second condition normalizes $R_\alpha$.   Because
\[
\mathcal{L}_{R_\alpha} \alpha = d  \iota_{R_\alpha}\alpha + \iota_{R_\alpha}  d\alpha
\]
the flow of  $R_\alpha$ preserves the form $\alpha$ and hence the contact structure $\xi$.  Note that if one chooses a different contact form $f \alpha$, the corresponding vector field $R_{f\alpha}$ is very different from $R_\alpha$, and its flow may have quite different properties.  

 A {\bf{Reeb orbit}} $\gamma$ of period $T$ associated to $R_\alpha$ is defined to be a path $\gamma: \R/T\Z \to M$ given by an integral curve of $R_\alpha$. That is,
\[ 
\frac{d\gamma}{dt} = R_\alpha \circ \gamma(t), \quad \gamma(0) = \gamma(T).
\]  
Two Reeb orbits
\[
\gamma_1, \ \gamma_0 : \R/T\Z \to M
\]
are considered equivalent if they differ by reparametrization, i.e. precomposition with a translation of $\R/T\Z.$  

The $N$-fold cover $\gamma^N$ is defined to be the composition of $\gamma_\pm$ with $\R/NT\Z \to \R/T\Z$.  A 
\textbf{simple Reeb orbit} is one such that $\gamma: \R/T\Z \to M$ is injective.  

\begin{rem}
Since Reeb vector fields are autonomous, the terminology ``simple Reeb orbit $\gamma$" refers to the entire equivalence class of orbits, and likewise for its iterates.
\end{rem}

A Reeb orbit $\gamma$ is said to be {\bf{nondegenerate}} whenever the linearized return map 
\[ d(\varphi_T)_{\gamma(0)}: \xi_{\gamma(0)} \to \xi_{\gamma(T) = \gamma(0)}\] has no eigenvalue equal to 1. A {\bf{nondegenerate contact form}} is one whose Reeb orbits are all nondegenerate and hence isolated. Note that since the Reeb flow preserves the contact structure, the linearized return map is symplectic.

 Next we briefly review the canonical contact form on $S^3$ and its Reeb dynamics. 
\begin{example}[Canonical Reeb dynamics on the 3-sphere]
\label{3-sphere}
{  If we define the following function $f\colon \R^4 \to \R$ 
\[
f(x_1, y_1, x_2, y_2)= x_1^2+y_1^2+x_2^2+ y_2^2,
\]  
then $S^3=f^{-1}(1)$. Recall that the canonical contact form on $S^3 \subset \R^4$ is given to be
\begin{equation}
\label{ls}
\lambda_0 := - \frac{1}{2} df \circ J = \left(  x_1  dy_1 - y_1 dx_1 +  x_2  dy_2 -  y_2  dx_2 \right)\arrowvert_{S^3}.
\end{equation}
The Reeb vector field is given by
\begin{equation}\label{reebreal}
\begin{array}{lcl}
R_{\lambda_0}&=&\left(x_1 \dfrac{\partial}{\partial y_1} - y_1 \dfrac{\partial}{\partial x_1} + x_2 \dfrac{\partial}{\partial y_2} - y_2 \dfrac{\partial}{\partial x_2}\right) \\
&=& (-y_1,x_1,-y_2,x_2). \\
\end{array}
\end{equation}
Equivalently we may reformulate these using complex coordinates by identifying $\R^4$ with $\C^2$ via
\[
u = x_1+iy_1, \ \ \ v = x_2+iy_2.
\]
We obtain
\[
\lambda_0=\frac{i}{2}\left(u d\bar{u} - \bar{u} du + v d\bar{v} - \bar{v} dv\right)\big |_{S^3},
\]
and
\begin{equation}
\label{reeb3sphere2}
\begin{array}{ccl}
R_{\lambda_0} & =& i \left( u \dfrac{\partial}{\partial u} -  \bar{u} \dfrac{\partial}{\partial \bar{u}} +  v \dfrac{\partial}{\partial v} -  \bar{v} \dfrac{\partial}{\partial \bar{v}}\right) \\
&=& (iu, iv) \\
\end{array}
\end{equation}
The second expression for $R_{\lambda_0}$ follows from (\ref{reebreal}) since $iu=(-y_1,x_1)$ and $iv=(-y_2,x_2)$.

To see that the orbits of $R_{\lambda_0}$ define the fibers of the Hopf fibration recall that a fiber through a point
\[
(u,v)=(x_1+iy_1, x_2+ iy_2) \in S^3 \subset \C^2,
\]
can be parameterized as 
\begin{equation}
\label{reebflow}
\varphi(t)=(e^{it}u, e^{it}v), \ t\in \R.
\end{equation}
We compute the time derivative of the fiber
\[
 \dot{\varphi}(0)=(iu,iv)=(i x_1 - y_1, i x_2 - y_2).
\]
Expressed as a real vector field on $\R^4$, which is tangent to $S^3$, this is the Reeb vector field $R_{\lambda_0}$ as it appears in (\ref{reeb3sphere2}), so the Reeb flow does indeed define the Hopf fibration.
}
\end{example}

\subsection{Hypersurfaces of contact type}\hspace{\fill}\\
Another notion that we need from symplectic and contact geometry is that of a hypersurface of contact type in a symplectic manifold.   The following notion of a Liouville vector field allows us to define hypersurfaces of contact type. Liouville vector fields will be used to understand the Reeb dynamics of the nondegenerate contact form $\alpha_1$ as well as to construct the contactomorphism between $(L_{A_n},\xi_0)$ and $(L(n+1,n),\xi_{std})$.

 \begin{definition}
\label{lioudef}
A \textbf{Liouville vector field} $Y$ on a symplectic manifold $(W, \omega)$ is a vector field satisfying
\[ \mathcal{L}_Y \omega = \omega \]
The flow $\psi_t$ of such a vector field is conformal symplectic, i.e. $\psi^*_t(\omega)=e^t \omega$.  The flow of these fields are volume expanding, so such fields may only exist locally on compact manifolds. 
\end{definition}

Whenever there exists a Liouville vector field $Y$ defined in a neighborhood of a compact hypersurface $Q$ of $(W, \omega)$, which is transverse to $Q$, we can define a contact 1-form on $Q$ by 
\[
\alpha : = \iota_Y\omega.
\]

\begin{prop}[{\cite[Prop 3.58]{MD}}]\label{contacttype}
Let $(W, \omega)$ be a symplectic manifold and $Q \subset W$ a compact hypersurface.  Then the following are equivalent: 
\begin{itemize}
\item[{(i)}] There exists a contact form $\alpha$ on $Q$ such that $d \alpha = \omega|_Q$.
\item[{(ii)}]  There exists a Liouville vector field $Y:U \to TW$ defined in a neighborhood $U$ of $Q$, which is transverse to $Q$. 
\end{itemize}
If these conditions are satisfied then $Q$ is said to be of \textbf{contact type.}
\end{prop}

We will need the following application of Gray's stability theorem to hypersurfaces of contact type to prove Theorem \ref{lenslinkcontacto} in Section \ref{sectionlinklens}.

\begin{lem}\cite[Lemma 2.1.5]{G}\label{graycor}
Let $Y$ be a Liouville vector field on a symplectic manifold $(W,\omega)$.  Suppose that $M_1$ and $M_2$ are hypersurfaces of contact type in $W$.   Assume that there is a smooth function 
\begin{equation}\label{heq}
h:W \to \R
\end{equation}
such that the time-1 map of the flow of $hY$ is a diffeomorphism from $M_1$ to $M_2$.  Then this diffeomorphism is in fact a contactomorphism from $(M_1, \ker \iota_Y \omega|_{TM_1})$ to $(M_2, \ker \iota_Y \omega|_{TM_2})$.
\end{lem}

\subsection{Symplectization} \hfill \\
The symplectization of a contact manifold is an important notion in defining Floer theoretic theories like symplectic and contact homology.  It will also used in our calculation of the Conley-Zehnder index.  Let $(M, \xi = \ker \alpha)$ to be a contact manifold.  The \textbf{symplectization} of $(M,\xi = \ker \alpha)$ is given by the manifold  $\R \times M$ and symplectic form
\[
\omega = e^t(d\alpha - \alpha \wedge dt) = d (e^t\alpha).
\]
Here $t$ is the coordinate on $\R$, and it should be noted that $\alpha$ is interpreted as a 1-form on $\R \times M$, as we identify $\alpha$ with its pullback under the projection $\R \times M \to M$.  

Any contact structure $\xi$ may be equipped with a complex structure ${J}$ such that $(\xi, {J})$ is a complex vector bundle. 
 This set is nonempty and contractible.    There is a unique canonical extension of the almost complex structure ${J}$ on $\xi$ to an $\R$-invariant almost complex structure $\tilde{J}$ on $T(\R \times M)$, whose existence is due to the splitting,
\begin{equation}
\label{decomp}
T(\R \times M) = \R \frac{\partial}{\partial t} \oplus \R R_{\alpha} \oplus \xi.
\end{equation}

\begin{definition}[Canonical extension of ${J}$ to $\tilde{J}$ on $T(\R \times M)$]\label{complexstruc}
Let $[a,b;v]$ be a tangent vector where $a, \ b \in \R$ and $v \in \xi$.  We can extend ${J}: \xi \to \xi$ to $\tilde{J}: T(\R \times M) \to T(\R \times M)$ by
\[
\tilde{J}[a,b;v] = [-b,a,{J}v].
\]
Thus $\tilde{J}|_\xi = {J}$ and $\tilde{J}$ acts on $\R  \frac{\partial}{\partial t} \oplus \R R_{\alpha}$ in the same manner as multiplication by $i$ acts on $\C$, namely ${J}  \frac{\partial}{\partial t} = R_{\alpha}$. 
\end{definition}


\subsection{The Conley-Zehnder index}\label{CZsection}\hfill \\
The Conley-Zehnder index $\mu_{CZ}$, is a Maslov index for arcs of symplectic matrices which  assigns an integer $\mu_{CZ}(\Phi)$ to every path of symplectic matrices $\Phi : [0,T] \to \mbox{Sp}(n)$, with $\Phi(0) = \mathds{1} $.   In order to ensure that the Conley-Zehnder index assigns the same integer to homotopic arcs, one must also stipulate that 1 is not an eigenvalue of the endpoint of this path of matrices, i.e. $\det(\mathds{1} - \Phi(T))\neq 0$.  We define the following set of continuous paths of symplectic matrices that start at the identity and end on a symplectic matrix that does not have 1 as an eigenvalue. 
\[
\Sigma^*(n) = \{ \Phi :[0,T] \to \mbox{Sp}(2n)  \ | \ \Phi \mbox{ is continuous},  \ \Phi(0)=\mathds{1}, \mbox{ and }  \mbox{det}(\mathds{1} - \Phi(T)) \neq 0  \}.
\]

The Conley-Zehnder index is a functor satisfying the following properties, and is uniquely determined by the homotopy, loop, and signature properties.

\begin{thm}\label{CZpropthm}{\cite[Theorem 2.3, Remark 5.4]{RS}}, {\cite[Theorem 2, Proposition 8 \& 9]{GuCZ}}\label{CZprop} \\
There exists a unique functor $\mu_{CZ}$ called the {\bf{Conley-Zehnder index}} that assigns the same integer to all homotopic paths $\Psi$ in $\Sigma^*(n)$,
\[
\mu_{CZ}: \Sigma^*(n) \to \Z.
\]
 such that the following hold.
\begin{enumerate}[\em (1)]
\item {\bf{Homotopy}}: The Conley-Zehnder index is constant on the connected components of $\Sigma^*(n)$.
\item {\bf{Naturalization}}: For any paths $\Phi, \Psi: [0,1] \to Sp(2n)$, $\mu_{CZ}(\Phi\Psi\Phi^{-1}) = \mu_{CZ}(\Psi)$.
\item {\bf{Zero}}: If $\Psi(t) \in \Sigma^*(n)$ has no eigenvalues on the unit circle for $t >0$, then $\mu_{CZ}(\Psi) = 0$.
\item {\bf{Product}}: If $n = n' + n''$, identify $Sp(2n') \oplus Sp(2n'')$ with a subgroup of $Sp(2n)$ in the obvious way. For $\Psi' \in \Sigma^*(n')$, $\Psi'' \in \Sigma^*(n'')$, then $\mu_{CZ}(\Psi' \oplus \Psi'') = \mu_{CZ}(\Psi') + \mu_{CZ}(\Psi'')$.
\item {\bf{Loop}}: If $\Phi$ is a loop at $\mathds{1}$, then $\mu_{CZ}(\Phi\Psi) = \mu_{CZ}(\Psi) + 2\mu(\Phi)$ where $\mu$ is the Maslov Index.
\item {\bf{Signature}}: If $S \in M(2n)$ is a symmetric matrix with $||S|| < 2\pi$ and $\Psi(t) = \exp(J_0St)$, then $\mu_{CZ}(\Psi) = \frac{1}{2}\sgn(S)$.
\end{enumerate}
\end{thm}

The linearized Reeb flow of $\gamma$ yields a path of symplectic matrices
 \[
  d(\varphi_t)_{\gamma(0)}: \xi_{\gamma(0)} \to \xi_{\gamma(t) = \gamma(0)}
  \]   
for $t\in[0,T],$ where $T$ is the period of $\gamma$.   

Thus we can compute the Conley-Zehnder index of $d\varphi_t, \ t\in[0,T].$ This index is typically dependent on the choice of trivialization $\tau$ of $\xi$ along $\gamma$ which was used in linearizing the Reeb flow. However, if $c_1(\xi;\Z)=0$ we can use the existence of an (almost) complex volume form on the symplectization to obtain a global means of linearizing the flow of the Reeb vector field. The choice of a complex volume form is parametrized by $H^1(\R \times M;\Z)$, so an absolute integral grading is only determined up to the choice of volume form.  See also \cite[\S 1.1.1]{jocompute}.

We define
\[
\mu_{CZ}^\tau(\gamma):=\mu_{CZ}\left( \left\{ d\varphi_t \right\}\arrowvert_{t\in[0,T]}\right)
\]
In the case at hand we will be able to work in the ambient space of $(\C^3, J_0)$, and use a canonical trivialization of $\C^3$.

\subsection{The canonical contact structure on Brieskorn manifolds}\hspace{\fill} \\
The $A_n$ link is an example of a Brieskorn manifold, which are defined generally by 
\[ 
\Sigma(\mathbf{a})= \left\{ (z_0,\dots,z_m) \in \C^{m+1} \ \bigg| \ f:= \sum_{j = 0}^m z_j^{a_j} = 0, \  a_j \in \Z_{>0} \text{ and } \sum_{j = 0}^m |z_j|^2 = 1 \right\}.
\]
The link of the $A_n$ singularity after a linear change of variables is $ \Sigma(n+1,2,2)$ for $n >3$; see (\ref{coorchange}).
Brieskorn gave a necessary and sufficient condition on $\mathbf{a}$ for $\Sigma(\mathbf{a})$ to be a topological sphere, and means to show when these yield exotic differentiable structures on the topological $(2n-1)$-sphere in \cite{Br}.   A standard calculus argument \cite[Lemma 7.1.1]{G} shows that $\Sigma(\mathbf{a})$ is always a smooth manifold.  

In the mid 1970's, Brieskorn manifolds were found to admit a canonical contact structure, given by their set of complex tangencies,
\[ 
\xi_0=T\Sigma \cap J_0 (T\Sigma), 
\]
where $J_0$ is the standard complex structure on $\C^{m+1}$. The contact structure $\xi_0$ can be expressed as $\xi_0 = \ker \alpha_0$ for the canonical 1-form   
\[
\alpha_0:= (- d\rho \circ J_0)|_\Sigma = \frac{i}{4} \left( \sum_{j=0}^m ( z_j d\bar{z}_j -\bar{z}_jdz_j ) \right)\bigg \vert_\Sigma,
\]
 where $\rho=(||z||^2-1)/4$.  A proof of this fact may be found in {\cite[Thm 7.1.2]{G}}. The Reeb dynamics associated to $\alpha_0$ are difficult to understand.  There is a more convenient contact form $\alpha_1$ constructed by Ustilovsky \cite[Lemma 4.1.2]{U} via the following family.  
 
\begin{prop}[{\cite[Proposition 7.1.4]{G}}] The 1-form 
\[
\alpha_t = \frac{i}{4}\sum_{j = 0}^m \frac{1}{1 - t + \frac{t}{a_j}} (z_j d\bar{z}_j - \bar{z}_jdz_j)
\]
 is a contact form on $\Sigma(\mathbf{a})$ for each $t\in [0,1]$.
\end{prop}

Via Gray's stability theorem we obtain the following corollary.
\begin{cor}
For all $t \in (0,1]$,   $(\Sigma(\mathbf{a}), \ker \alpha_0)$ is contactomorphic to $(\Sigma(\mathbf{a}), \ker \alpha_t)$.
\end{cor}

Next we compute the Reeb dynamics associated to $\alpha_1 =  \frac{i}{4}\sum_{j=0}^m a_j (z_j d\bar{z}_j - \bar{z}_jdz_j) $.
\begin{rem}
While $\alpha_1$ is degenerate,  one can still easily check that the Reeb vector field associated to $\alpha_1$ is given by,
\[
R_{\alpha_1} = 2i \sum_{j = 0}^m \frac{1}{a_j}\left( z_j \frac{\partial}{\partial z_j} - \bar{z}_j \frac{\partial}{\partial \bar{z}_j} \right) = 2i \left( \frac{z_0}{a_0},...,\frac{z_m}{a_m} \right).
\]
Indeed, one computes
\[
df\left(R_{\alpha_1}\right) = f(\mathbf{z}) \mbox{ and } d\rho \left(R_{\alpha_1}\right) =0.
\]
This shows that $R_{\alpha_1}$ is tangent to $\Sigma(\mathbf{a})$.  The defining equations for the Reeb vector field are satisfied since
\[
\alpha_1\left(R_{\alpha_1}\right) \equiv 1 \mbox{ and } \iota_{R_{\alpha_1}}d\alpha_1 = -d\rho,
\]
with the latter form being zero on the $T_p\Sigma(\mathbf(a))$.  The flow of $R_{\alpha_1}$ is given by
\[
\varphi_t(z_0,...,z_m) = \left( e^{2it/a_0},...,e^{2it/a_m} \right)
\]
All the orbits of the Reeb flow are closed, and the flow defines an effective $S^1$-action on $\Sigma(\mathbf{a})$.
\end{rem}

In the next section we perturb $\alpha_1$ to a nondegenerate contact form.  


\section{Proof of Theorem \ref{CZcomputation}}\label{CZcomputationsection}
\subsection{Constructing a nondegenerate contact form} \hspace{\fill} \\
In this section we adapt a method used by Ustilovsky in \cite[Section 4]{U} to obtain a nondegenerate contact form $\alpha_\epsilon$ on $L_{A_n}$ whose kernel is contactomorphic to $\xi_0$.  Ustilovsky's methods yielded a nondegenerate contact form on Brieskorn manifolds of the form $\Sigma(p,2,....,2)$, which were diffeomorphic to $S^{4m+1}$. 

We define the following change of coordinates to go from  $\Sigma(n+1,2,2)$ with defining function $f=z_0^{n+1} + z_1^2+z_2^2$ to $L_{A_n}$ with defining function $f_{A_n}= w_0^{n+1} + 2w_1w_2.$

\begin{equation}\label{coorchange}
\Psi(w_0,w_1,w_2) = \left(\underbrace{w_0}_{{:=z_0}} \ , \underbrace{\tfrac{\sqrt{2}}{2}(w_1+w_2)}_{:=z_1} \ , \underbrace{\tfrac{\sqrt{2}}{2}(-iw_1+iw_2)}_{:=z_2} \right)
\end{equation}
We obtain
\begin{align}
	\Psi^*f(z_0,z_1,z_2)= w_0^{n+1} + 2w_1w_2.
\end{align}
Then the pull-back of 
\[
\frac{\alpha_1}{2} =  \frac{i}{8}\sum_{j=0}^m a_j (z_j d\bar{z}_j - \bar{z}_jdz_j) 
\]
is given by
\[
\frac{\Psi^*\alpha_1}{2} = \frac{(n+1)i}{8}(w_0 d\overline{w}_0 - \overline{w}_0 dw_0) + \frac{i}{4}(w_1 d\overline{w}_1 - \overline{w}_1dw_1 + w_2 d\overline{w}_2 - \overline{w}_2 dw_2).\]
We now construct the Hamiltonian function 
\[H(w)=|w|^2+ \epsilon(|w_{1}|^2-|w_{2}|^2)\]
We choose $0<\epsilon<1$ such that $H(w)$ is positive on $S^5$, and define the  contact form 
\begin{equation}
\alpha_\epsilon= \frac{\Psi^*\alpha_1}{2H}
\end{equation}

\begin{rem}
The above shows that $(\Sigma(n+1,2,2), \ker \alpha_1)$   is contactomorphic to $(\Psi(\Sigma(n+1,2,2)), \ker \alpha_\epsilon)$.  Moreover $L_{A_n}=\Psi(\Sigma(n+1,2,2))$, where $L_{A_n}$ was defined in $(\ref{linkeq})$.
\end{rem}

\begin{prop}\label{perturbedreebprop} The Reeb vector field for $\alpha_\epsilon$ is 
\begin{align}\label{perturbedreeb}
	R_{\alpha_\epsilon} & =\frac{4i}{n+1}w_0\frac{\p}{\p w_0}-\frac{4i}{n+1}\overline{w}_0\frac{\p}{\p \overline{w}_0} +
	2i(1+\epsilon)\left(w_{1}\frac{\p}{\p w_{1}}-\overline{w}_{1}\frac{\p}{\p \overline{w}_{1}} \right)\notag  \\
	& + 2i(1 - \epsilon) \left(w_{2} \frac{\p}{\p w_{2}} -  \overline w_{2}\frac{\p}{\p \overline w_{2j}}\right) \notag \\
	& = \left( \frac{4i}{n+1}w_0,2i(1+\epsilon)w_1,2i(1-\epsilon)w_2\right). 
\end{align}
\end{prop}
\begin{rem}
The second formulation of the Reeb vector field is equivalent to the first in the above Proposition via the standard identification of $\R^4$ with $\C^2$, as explained in Example \ref{3-sphere}, equation (\ref{reeb3sphere2}).
\end{rem}

Before proving Proposition \ref{perturbedreebprop} we need the following lemma.
\begin{lem}\label{helper}
On $\C^3$, the vector field 
\begin{align}
	X(w) = \frac{1}{2}\left(\sum_{j=0}^{2} w_j\frac{\p}{\p w_j} + \overline{w}_j\frac{\p}{\p \overline{w}_j}\right)
\end{align}
is a Liouville vector field for the symplectic form
\[\omega_1=\frac{d(\Psi^*\alpha_1)}{2}=\frac{i(n+1)}{4}dw_0 \wedge d\overline{w}_0 + \frac{i}{2}\sum_{j=1}^{2} dw_j \wedge d\overline{w}_j.\]
The Hamiltonian vector field $X_H$ of $H$ with respect to $\omega_1$ is $-R_{\alpha_\epsilon}$, as in \emph{(\ref{perturbedreeb})}
\end{lem}

\begin{proof}
Recall that the condition to be a Liouville vector field is $\mathcal{L}_X \oo_1 = \oo_1$. We show this with Cartan's formula:
\begin{align*}
	\mathcal{L}_X\oo_1 & = \iota_X d\oo_1 + d(\iota_X \oo_1) \\
	& = d(\iota_X \oo_1).
\end{align*}
We do the explicit calculation for the first term and the rest easily follows: 
\begin{align*}
	d \left( \frac{i(n+1)}{4} d\oo_0 \wedge d\overline{\oo}_0 \left( \frac{1}{2} \left( w_0 \frac{\p}{\p w_0} + \overline{w}_0 \frac{\p}{\p \overline{w}_0}\right), \cdot \right)  \right) & = d \left( \frac{i(n+1)}{8} w_0 d\overline{w}_0 - \overline{w}_0 dw_0\right) \\
	& = \frac{i(n+1)}{8} \left( dw_0 \wedge d\overline{w}_0 - d\overline{w}_0 \wedge dw_0 \right) \\
	& = \frac{i(n+1)}{4} dw_0 \wedge d\overline{w}_0,
\end{align*}
so $X(w)$ is indeed a Liouville vector field for $\oo_1$. \\

Next we prove that $\oo_1(-R_{\alpha_\epsilon},\cdot) = dH(\cdot)$.  First we calculate $dH$,
\[ dH = \left(\sum_{j = 0}^{2} w_j  d\overline{w}_j + \overline{w}_j  dw_j\right) + \epsilon(w_{1} d\overline{w}_{1} + \overline{w}_{1} dw_1 - w_{2} d\overline{w}_{2} - \overline{w}_{2}dw_{2}).\]
Then we compare the coefficients of $dH$ to the coefficients of $\oo_1(-R_{\alpha_\epsilon},\cdot)$  associated to each term, $(dw_i \wedge d\overline{w}_i)$. The $(dw_0 \wedge d\overline{w}_0)$ term is
\begin{align*}
	\frac{i(n+1)}{4} dw_0 \wedge d\overline{w}_0 \left(-\frac{4i}{n+1} w_0 \frac{\p}{\p w_0}  + \frac{4i}{n+1} \overline{w}_0 \frac{\p}{\p \overline{w}_0} ,\cdot\right) & = \frac{i(n+1)}{4} \left(- \frac{4i}{n+1} w_0 d\overline{w}_0 - \frac{4i}{n+1} \overline{w}_0 dw_0 \right) \\
	& = w_0 d\overline{w}_0 + \overline{w}_0 dw_0.
\end{align*}
The  $(dw_{1} \wedge d\overline{w}_{1})$ term is
\begin{align*}
	\frac{i}{2} dw_{1} \wedge d\overline{w}_{1} \left( -2i(1 + \epsilon) w_{1} \frac{\p}{\p w_{1} } + 2i(1 + \epsilon) \overline{w}_{1} \frac{\p}{\p \overline{w}_{1}} \right) & = \frac{i}{2} \left( - 2i(1 + \epsilon)w_{1} d\overline{w}_{1} - 2i(1 + \epsilon)\overline{w}_{1}dw_{1}\right) \\
	& = (1 + \epsilon) w_{1} d\overline{w}_{1} + (1 + \epsilon) \overline{w}_{1} dw_{1}.
\end{align*}
The $(dw_2 \wedge d\overline{w}_2)$ term is obtained similarly. Summing the terms yields  $\oo_1(-R_{\alpha_\epsilon},\cdot) = dH(\cdot)$.
\end{proof}

\begin{proof}[Proof of Proposition \ref{perturbedreebprop}]
First we show that $X_H =-R_{\alpha_\epsilon}$ is tangent to the link $\Psi(\Sigma(n+1,2,2) )$. We compute
\begin{align*}
	 (\Psi_*df)(R_{\alpha_\epsilon}) =
	& = \left( (n+1)w_0^n dw_0 + 2w_{1}dw_{2} + 2w_{2} dw_1 \right) (R_{\alpha_\epsilon}) \\
	& = 4i w_0^{n+1} + 4i(1 - \epsilon) w_{1}w_{2} + 4i(1 + \epsilon) w_{1}w_{2} \\
	& = 4i (\Psi^*f) \\
	& = 0
\end{align*}
the last equality because $\Psi^*f$ is constant along $\Psi(\Sigma(n+1,2,2) )$. Now we have to show that $\dfrac{\Psi^*\alpha_1}{2}(X_H) = -H$. We have 
\begin{align*}
	\dfrac{\Psi^*\alpha_1}{2}\left(\cdot\right) & = \iota_X\oo_1(\cdot) = \oo_1(X(w),\cdot) = - \oo(\cdot,X(w)) \\
	\dfrac{\Psi^*\alpha_1}{2}(X_H) & = -\oo(X_H,X(w)) = - dH(X(w)) \\
	& = - |w|^2 - \epsilon (|w_{1}|^2 - |w_{2}|^2) \\
	& = -H.
\end{align*}
From these, we conclude
\begin{align*}
	\alpha_\epsilon(X_H) & = -\frac{1}{H}H = -1 \\
	d\alpha_\epsilon(X_H,\cdot) & = - \frac{1}{2H^2} (dH \wedge \Psi^* \alpha_1)(X_H,\cdot) + \frac{1}{2H} d\Psi^*\alpha_1(X_H,\cdot) \\
	& = - \frac{1}{2H^2} dH(X_H) \Psi^*\alpha_1(\cdot) + \frac{1}{2H^2} \Psi^*\alpha_1(X_H) dH(\cdot) + \frac{1}{H} \oo (X_H,\cdot) \\
	& = - \frac{1}{2H^2}\oo_1(X_H,X_H) \Psi^*\alpha_1(\cdot) - \frac{1}{H} dH(\cdot) + \frac{1}{H} dH(\cdot) \\
	& = 0
\end{align*}
By Lemma \ref{helper}, we know $-X_H = R_{\alpha_\epsilon}$ so the result follows. 
\end{proof}
\subsection{Isolated Reeb Orbits} \hspace{\fill} \\
 In this quick section, we prove the following proposition.
\begin{prop}
The only simple periodic Reeb orbits of $R_{\alpha_\epsilon}$ are nondegenerate and defined by
\begin{align*} 
\gamma_+(t) & = (0,e^{2i(1 + \epsilon)t},0), \quad \quad 0 \le t \le \frac{\pi}{1 + \epsilon} \\
\gamma_-(t) & = (0,0,e^{2i(1 - \epsilon)t}), \quad \quad 0 \le t \le \frac{\pi}{1 + \epsilon}.
\end{align*}
\end{prop}
\begin{proof}
The flow of
\[ R_{\alpha_\epsilon} = \left( \frac{4i}{n + 1}w_0,2i(1 + \epsilon)w_1,2i(1 - \epsilon)w_2\right)\]
is given by
\[\varphi_t(w_0,w_1,w_2) = \left(e^{\frac{4it}{n+1}}w_0,e^{2i(1+\epsilon)t}w_{1},e^{2i(1-\epsilon)t}w_{2}\right).\]
Since $\epsilon$ is small and irrational, the only possible periodic trajectories are
\begin{align*}
\gamma_0(t) & = (e^{\frac{4i}{n+1}t},0,0) \\
\gamma_+(t) & = (0,e^{2i(1 + \epsilon)t},0) \\
\gamma_-(t) & = (0,0,e^{2i(1 - \epsilon)t}).
\end{align*}

It is important to note that the first trajectory does not lie in $\Psi(\Sigma(n+1,2,2))$, but rather on total space $\C^3$. This is because the point $\gamma_0(0) = (1,0,0)$ is not a zero of $f_{A_n}=w_0^{n+1}+2w_1w_2$. 

Next we need to check that the linearized return maps $d\phi|_\xi$ associated to $\gamma_+$ and $\gamma_-$  have no eigenvalues equal to 1.  We consider the first orbit $\gamma_+$ of period $\pi/(1 + \epsilon)$, as a similar argument applies to the return flow associated to $\gamma_-$. The differential of its total return map is:
\[ d\varphi_{T} = \left. \begin{pmatrix}
e^{\frac{4iT}{n+1}} & 0 & 0 \\
0 & 1 & 0 \\
0 & 0 & e^{2i(1 - \epsilon)T} 
\end{pmatrix}\right\arrowvert_{T=\frac{\pi}{1+\epsilon}} \] 
Since $\epsilon$ is a small irrational number, the total return map only has one eigenvalue which is 1. The eigenvector associated to the eigenvalue which is 1 is in the direction of the Reeb orbit $\gamma^+$, but since we are restricting the return map to $\xi$, we can conclude that $\gamma_+$ is nondegenerate.
\end{proof}





\subsection{Computation of the Conley-Zehnder index}\hspace{\fill} \\

To compute the Conley-Zehnder indices of the Reeb orbits in Theorem 1.1 we use the same method as in \cite{U}, extending the Reeb flow to give rise to a symplectomorphism of $\C^3\setminus \{\mathbf{0} \}$. This permits us to do the computations in $\C^3$, equipped with the symplectic form 
\[
\omega_1=\frac{d(\Psi^*\alpha_1)}{2}=\frac{i(n+1)}{4}dw_0 \wedge d\overline{w}_0 + \frac{i}{2}\sum_{j=1}^{2} dw_j \wedge d\overline{w}_j.
\]
We may equip the contact structure $\xi_0$ with the symplectic form $\oo = d\alpha_1$ instead of $d\alpha_\epsilon$ when computing the Conley-Zehnder indices. This is because $ \ker \alpha_\epsilon = \ker \alpha_1 = \xi_0$, as $\alpha_\epsilon = \frac{1}{H} \alpha_1$ with $H > 0$ and because $\oo|_\xi = Hd\alpha_\epsilon|_\xi$ and $H$ is constant along Reeb trajectories.    

Our first proposition shows that we can construct a standard symplectic basis for the symplectic complement
\[
\xi^\oo = \{ v \in \C^3 \ | \ \oo(v,w) = 0 \text{ for all $w \in \xi$}\}
\]
of $\xi$ in $\C^3$.  As a result,  $c_1(\xi^\oo)=0$. Since $c_1(\C^3)=0$, we know $c_1(\xi)=0$.  Thus we may compute the Conley-Zehnder indices in the ambient space $\C^3$ and use additivity of the Conley-Zehnder index under direct sums of symplectic paths to compute it in $\xi$.

\begin{prop}
There exists a standard symplectic basis for the symplectic complement $\xi^\oo$ with respect to  $\oo = d\alpha_1$. 
\end{prop}
\begin{proof}
Notice that $\xi^\oo = \mbox{span}(X_1, Y_1,X_2,Y_2)$ where
\begin{align*}
X_1 & = (\bar{w}_0^n,\bar{w}_1,\bar{w}_2) \quad Y_1  = iX_1 \\
X_2 & = R_\epsilon \quad \quad \quad \quad \quad Y_2  = w.
\end{align*}
 We make this a into a symplectic standard basis for $\xi^\oo$ via a Gram-Schmidt process.  The new basis is given by:
\[
\begin{array}{rclc rcl}
\tilde X_1 & = & \dfrac{X_1}{\sqrt{\oo(X_1,Y_2)}} & \ \ \ \ \ \ & \tilde Y_1 & =& \dfrac{Y_1}{\sqrt{\oo(X_1,Y_1)}} = i \tilde X_1 \\
&&&&&& \\
 \tilde X_2 &=& X_2 &\ \ \ \ \ \ & \tilde Y_2 &= & Y_2 - \dfrac{\oo(X_1,Y_2)Y_1 - \oo(Y_1,Y_2)X_1}{\oo(X_1,Y_1)}  \\
 &&&&&& \\
  &&&& \ \ \ \ \ \  &= &Y_2 - \dfrac{n-1}{2}w_0^{n+1}{w(X_1,Y_1)}X_1.  \\
\end{array}
\]
This is a standard basis for the symplectic vector space $\xi^\oo$, i.e. the form $\oo$ in this basis is given by 
\[\begin{pmatrix}
	\begin{pmatrix}
	0 & 1 \\
	1 & 0 
	\end{pmatrix} & \\
	& \begin{pmatrix}
	0 & 1 \\
	1 & 0
	\end{pmatrix}
\end{pmatrix}.\]
\end{proof}


Now we are ready to prove the Conley-Zehnder index formula in Theorem \ref{CZcomputation}.

\begin{prop}
The Conley-Zehnder index for $\gamma = \gamma_{\pm}^N$ in $0 \le t \le \frac{N\pi}{1 \pm \epsilon}$ is
\begin{align}
\mu_{CZ}(\gamma_{\pm}^N) = 2\left( \left\lfloor \frac{2N}{(n+1)(1 \pm \epsilon)}\right \rfloor + \left\lfloor \frac{N(1 \mp \epsilon)}{1 \pm \epsilon} \right\rfloor - \left \lfloor \frac{2N}{1 \pm \epsilon} \right \rfloor \right) + 2N + 1.
\end{align}
\end{prop}
\bigskip

\begin{proof}

The Reeb flow $\varphi$ which we introduced in the previous section can be extended to a flow on $\C^3 $, which we also denote by $\varphi$.  The action of the extended Reeb flow on $\C^3$ is given by:
\[
\begin{array}{lclclcl}
d\varphi_t(w)\tilde X_1 & = & e^{4it}\tilde X_1(\varphi_t(w)) & \ \ \ &d\varphi_t(w)\tilde Y_1 & = & e^{4it}\tilde Y_1(\varphi_t(w)) \\
d\varphi_t(w)\tilde X_2 & = &\tilde X_2(\varphi_t(w)) &  \ \ \ &d\varphi_t(w)\tilde Y_2 &=& \tilde Y_2(\varphi_t(w)). \\
\end{array}
\]
 Define
\[
\Phi  := d\varphi_t \big |_{\C^3}  = \diag \left(e^{\frac{4i}{n+1}t},e^{2i(1+\epsilon)t},e^{2i(1 - \epsilon)t}\right) 
\]
We can now use the additivity of the Conley-Zehnder index under direct sums of symplectic paths, Theorem \ref{CZprop}(4) to get 
\[
\mu_{CZ}(\gamma_\pm) = \mu_{CZ}(\Phi) - \mu_{CZ}(\Phi_{\xi^\oo}),
\]
where
\begin{equation}\label{CZsum}
\Phi_{\xi^\oo}  := d\varphi_t \big |_{\xi^\oo}  = \diag \left(e^{4it},1\right). 
\end{equation}
The right hand side of (\ref{CZsum}) is easily computed via the crossing form; see \cite[Rem 5.4]{RS}.  In particular we have
\[
\mu_{CZ}\left(\{e^{it}\} \big |_{t\in [0,T]}\right) = \left\{ \begin{array}{ll}
	\dfrac{T}{\pi}, & T \in 2\pi \Z \\
	&\\
	2 \left \lfloor \dfrac{T}{2\pi} \right \rfloor + 1,\ \ \  & \text{otherwise.}
	\end{array} \right.
\]  

\medskip

\noindent Thus for $\{ \Phi(t) \} = \{ e^{4it/(n+1)}\oplus e^{2it(1 + \epsilon)} \oplus e^{2it(1 - \epsilon)} \}$ with $0 \leq t \leq T$ we obtain: \\

\begin{align*}
\mu_{CZ}(\Phi) & = \left\{ \begin{array}{ll}
\dfrac{4T}{(n+1)\pi}, & T \in \frac{(n+1)\pi}{2}\Z \\
&\\
2 \left\lfloor \dfrac{2T}{(n+1)\pi}\right \rfloor + 1,\ \ \  & T \notin \frac{(n+1)\pi}{2}\Z
\end{array}\right. \ \ \ \ \ +  \ \ \ \left\{ \begin{array}{ll}
\dfrac{2T(1+\epsilon)}{\pi}, & T \in \frac{\pi}{1 + \epsilon}\Z \\
&\\
2 \left\lfloor \dfrac{T(1 + \epsilon)}{\pi} \right\rfloor + 1,\ \ \  & T \notin \frac{\pi}{1 + \epsilon}\Z
\end{array}\right. \\
&\\
& + \left\{ \begin{array}{ll}
\dfrac{2T(1-\epsilon)}{\pi}, & T \in \frac{\pi}{1 - \epsilon}\Z \\
&\\
2 \left \lfloor \dfrac{T(1 - \epsilon)}{\pi}\right \rfloor + 1,\ \ \  & T \notin \frac{\pi}{1 - \epsilon}\Z.
\end{array}\right. 
\end{align*}

\medskip

\noindent Likewise for $\Phi_{\xi^\oo}$ with $0 \leq t \leq T$ we obtain:
\begin{align*}
\mu_{CZ}(\Phi_{\xi^\oo}) & =  \left\{ \begin{array}{ll}
\dfrac{4T}{\pi}, & T \in \frac{\pi}{2}\Z \\
&\\
2 \left \lfloor \dfrac{2T}{\pi} \right \rfloor + 1,\ \ \  & T \notin \frac{\pi}{2}\Z.
\end{array}\right.
\end{align*}
Hence we get that the Conley-Zehnder index for $\gamma_{\pm}^N$ in $0 \le t \le \frac{N\pi}{1 \pm \epsilon}$ is given by:
\begin{equation}
\mu_{CZ}(\gamma_{\pm}^N) = 2\left( \left\lfloor \frac{2N}{(n+1)(1 \pm \epsilon)}\right \rfloor + \left\lfloor \frac{N(1 \mp \epsilon)}{1 \pm \epsilon} \right\rfloor - \left \lfloor \frac{2N}{1 \pm \epsilon} \right \rfloor \right) + 2N + 1.
\end{equation}
\end{proof}


\section{Proof of Theorem \ref{lenslinkcontacto}}\label{sectionlinklens}
This section proves that $(L_{A_n},\xi_0)$ and $(L(n+1,n),\xi_{std})$ are contactomorphic.  This is done by constructing by constructing a 1-parameter family of contact manifolds via a canonically defined Liouville vector field and applying Gray's stability theorem.  

\subsection{Contact geometry of $(L(n+1,n),\xi_{std})$}\label{contactgeomlens} \hfill \\
The lens space $L(n+1,n)$ is obtained via the quotient of $S^3$ by the binary cyclic subgroup $A_n \subset SL(2,\C)$.  The subgroup $A_n$ is given by the action of $\Z_{n+1}$ on $\C^2$ defined by
\begin{align*} \begin{pmatrix}
u \\
v 
\end{pmatrix} \mapsto \begin{pmatrix}
e^{2\pi i/(n+1)} & 0 \\
0 & e^{2 n\pi i/(n+1)}
\end{pmatrix}\begin{pmatrix}
u \\
v
\end{pmatrix} . \\
\end{align*}

The following exercise shows that $L(n+1,n)$ is homeomorphic to $L_{A_n}$.
This construction will be needed later on in another proof, so we explain it here to set up the notation.  

The origin is the only fixed point of the $A_n$ action on $\C^2$ and hence is an isolated quotient singularity of $\C^2/A_n$.   We can represent $\C^2/A_n$ as a hypersurface of $\C^3$ as follows.  Consider the monomials
\[ 
z_0 := uv, \quad z_1 := \tfrac{i}{\sqrt{2}}u^{n+1}, \quad z_2 := \tfrac{i}{\sqrt{2}}v^{n+1}
.\]
These are invariant under the action of $A_n$ and satisfy the equation $z_0^{n+1} + 2z_1z_2  = 0$. Recall that 
\[
 f_{A_n}(z_0,z_1,z_2) = z_0^{n+1} + 2z_1z_2, 
 \] 
 and 
 \[
 L_{A_n}=S^5 \cap \{ f_{A_n}^{-1}(0) \}
 \]
 
Moreover,
\begin{equation}\label{varphieq}
\begin{array}{llcl}
\tilde\varphi: &\C^2 &\to& \C^3 \\
&(u,v) &\mapsto &(uv,\tfrac{i}{\sqrt{2}}u^{n+1},\tfrac{i}{\sqrt{2}}v^{n+1})\\
\end{array}
\end{equation} descends to the map 
\[
\varphi: \C^2/A_n \to \C^3, 
\]
which sends $\varphi (\C^2/A_n)$ homeomorphically onto the hypersurface $f^{-1}_{A_n}(0)$.

Rescaling away from the origin of $\C^3$ yields a homeomorphism between $\varphi(S^3/A_n)$ and  $L_{A_n}$. As 3-manifolds which are homeomorphic are also diffeomorphic \cite{moise} we obtain the following proposition.

\begin{prop}
$L(n+1,n)$ is diffeomorphic to $L_{A_n}$.
\end{prop}

\begin{rem}
 In order to prove that two manifolds are contactomorphic, one must either construct an explicit diffeomorphism or make use of Gray's stability theorem.  Sadly, $\varphi$ is not a diffeomorphism onto its image when $u=0$ or $v=0$.   As the above diffeomorphism is only known to exist abstractly, we will need to appeal the latter method to prove that  $(L_{A_n},\xi_0)$ and $(L(n+1,n),\xi_{std})$ are contactomorphic.  As a result, this proof is rather involved.  \end{rem}

Our application of Gray's stability theorem uses the flow of a Liouville vector field to construct a 1-parameter family of contactomorphisms.   First we prove that $L(n+1,n)$ is a contact manifold whose contact structure descends from the quotient of $S^3$. 

Consider the standard symplectic form on $\C^2$ given by 
\begin{equation}
\begin{array}{lcl}
\omega_{\C^2}&=&d\lambda_{\C^2} \\
\lambda_{\C^2} &= &\dfrac{i}{2} \left(u d\bar u - \bar u du + v d\bar v - \bar v dv\right). \\
\end{array}
\end{equation}
The following proposition shows that $\lambda_0$ restricts to a contact form on $L(n+1,n)$. We define $\ker\lambda = \xi_{\text{std}}$ on $L(n+1,n)$. 

 \begin{prop}\label{calcliou}
The vector field
\[ 
Y_0 = \frac{1}{2} \left( u\frac{\p}{\p u} + \bar{u} \frac{\p }{\bar{u}} + v\frac{\p}{\p v} + \bar{v} \frac{\p }{\bar{v}} \right)
\]
is a Liouville vector field on $(\C^2/A_n,\omega_{\C^2})$ away from the origin and transverse to $L(n+1,n)$.
\end{prop}
\begin{proof}
We have that $\C^2/A_n$ is a smooth manifold away from the origin because $0$ is the only fixed point by the action of $A_n$. Write 
\[ 
S^3/A_n = \{ (u, v) \in \C^2/A_n\ \big | \ |u|^2+|v|^2 = 1\}.
\]
Then $L(n+1,n) =S^3/A_n$ is a regular level set of $g(u,v) = |u|^2+|v|^2$ Choose a Riemannian metric on $\C^2/A_n$ and note that
\[ 
Y_0 = \frac{1}{4} \nabla g.
\] 
Thus $Y_0$ is transverse to $L(n+1,n)$. Since
\[
 \mathcal{L}_{Y_0} \omega_{\C^2}  = d(i_{Y_0}d\lambda_{\C^2}) = \omega_{\C^2},
 \]
we may conclude that $Y_0$ is indeed a Liouville vector field on $(\C^2/A_n, \omega_{\C^2})$ away from the origin. Thus by Proposition \ref{contacttype}, $L(n+1,n)$ is a hypersurface of contact type in $\C^2/A_n$.
\end{proof}


\subsection{ The proof that $(L_{A_n},\xi_0)$ and $(L(n+1,n),\xi_{std})$ are contactomorphic} \hfill \\

First we set up $L_{A_n}$ and $\varphi(L(n+1,n))$ as hypersurfaces of contact type in $\{f^{-1}_{A_n}(0) \} \setminus \{ \mathbf{0} \}$.  Define $\rho : \C^3 \to \R$ by
\[ \rho(z) = \frac{|z|^2 - 1}{4} = \frac{z_0\bar{z}_0 + \cdots + z_2\bar{z}_2 - 1}{4}.
\] 
The standard symplectic structure on $\C^3$ is given by.
\[ \oo_{\C^3} = \frac{i}{2}( dz_0\wedge d\bar{z}_0 + \cdots + dz_2 \wedge d\bar{z}_2).\]
Moreover, 
\begin{equation}\label{liouY}
 Y = \nabla \rho =  \frac{1}{2} \sum_{j = 0}^2 z_j \frac{\p}{\p z_j} + \bar{z}_j \frac{\p}{\p \bar{z}_j}
 \end{equation} 
 is a Liouville vector field for $(\C^3,\oo_{\C^3}).$  We define
 \[
\lambda_{\C^3}= \iota_Y \omega_{\C^3}.
  \]
A standard calculation analogous to the proof of Proposition \ref{calcliou} shows that $Y$ is a Liouville vector field on $\left( \{f^{-1}_{A_n}(0) \} \setminus \{ \mathbf{0} \}, \omega_{\C^3} \right)$

 \begin{rem}
Both $\varphi(L(n+1,n))$ and $L_{A_n}$ are  hypersurfaces of contact type in 
$\left( \{f^{-1}_{A_n}(0) \} \setminus \{ \mathbf{0} \}, \omega_{\C^3} \right)$.  
Note that $\varphi(L(n+1,n))$ is in fact transverse to the Liouville vector field $Y$ because
\[  
  \begin{array}{ccl}
\varphi(L(n+1,n)) &=& \varphi \left( \left \{|u|^{2} + |v|^{2} = 1 \right \}/ A_n\right) \\
&=& \varphi ( \{|u|^{4} + 2|u|^2|v|^2+ |v|^{4} = 1 \}/ A_n) \\
 &=&\left \{ 2|z_0|^2+4^{1/(n+1)}|z_1|^{4/(n+1)} +  4^{1/(n+1)} |z_2|^{4/(n+1)} = 1 \right \} \cap f_{A_n}^{-1}(0)  \\
  \end{array} \]
\end{rem}

We will want $\varphi(L(n+1,n))$ and $L_{A_n}$ to be disjoint in $\{f^{-1}_{A_n}(0) \}. $  This is easily accomplished by rescaling $r$ in the definition of the link.  

\begin{definition}
Define
\[
L_{A_n}^r = f^{-1}_{A_n}(0) \cap S^5_r,
\]
with the assumption that $r$ has been chosen so that $\varphi(L(n+1,n))$ and $L_{A_n}^r$ are  disjoint in $\{f^{-1}_{A_n}(0) \}$ and so that the flow of the Liouville vector field $Y$ ``hits" $\varphi(L(n+1,n))$ before $L_{A_n}^r$. 
\end{definition}

The first result is the following lemma, which provides a 1-parameter family of diffeomorphic manifolds starting on $\varphi(L(n+1,n))$ and ending on $L_{A_n}^r$.  First we set up some notation.  Let
\[
\psi_t:\R \times X \to X
\]
 be the flow of $Y$ and  $\psi_t(z) =  \gamma_z(t)$ the unique integral curve passing through $z \in \varphi(L(n+1,n))$ at time $t = 0$.  For any integral curve $\gamma$ of $Y$ we consider the following initial value problem:
\begin{equation}
\label{ivp}
\begin{array}{ccl}
\gamma'(t)&=&Y(\gamma(t))\\
\gamma(0)&=&z \in  \varphi(L(n+1,n)) \\
\end{array}
\end{equation}
By means of the implicit function theorem and the properties of the Liouville vector field $Y$ we can prove the following claim.

\begin{lem}\label{oneparameter}
For every $\gamma_z$, there exists a $\tau(z) \in \R_{> 0}$ such that $\gamma_z(\tau(z)) \in L_{A_n}^r$.  The choice of $\tau(z) $ varies smoothly for each $z \in \varphi(L(n+1,n))$.
\end{lem}




\begin{proof}
In order to apply the implicit function theorem, we must show for all $(t,z)$ with $\rho \circ \gamma =0$ that
\[
\frac{\p (\rho \circ \gamma)}{\p t} \neq 0.
\]
Note that $\rho \circ \gamma$ is smooth.  By the chain rule,
\[
\left. \frac{\p (\rho \circ \gamma)}{\p t}\right|_{(s,p)} = \mbox{grad }\rho |_{\gamma(s,p)} \cdot \dot{\gamma}|_{(s,p)},
 \]
where $\dot{\gamma}|_{(s,p)} = \frac{\p \gamma}{\p t}|_{(s,p)}$. 

If $\mbox{grad } \rho \arrowvert_{\gamma(s,p)} \cdot \dot{ \gamma}|_{(s,p)} = 0$, then either $\mbox{grad } \rho$ is not transverse along $\{ (\rho \circ \gamma) \ (s,p)=0 \}$ or  $ \dot{ \gamma}|_{(s,p)} = 0$, since $\mbox{grad } \rho \neq 0$.  By construction grad $\rho = \nabla \rho$ is a Liouville vector field transverse to $L_{A_n}^r$ .  Furthermore, the conformal symplectic nature of a Liouville vector field implies that for any integral curve $\gamma$ satisfying the initial value problem given by equation (\ref{ivp}), $\dot{\gamma}|_{(s,p)} \neq 0$.  Thus we see that the conditions for the implicit function theorem are satisfied and our claim is proven. 


\end{proof}

\begin{rem}\label{helperrem}
The time $\tau(z)$ can be normalized to 1 for each $z$, yielding a 1-parameter family of diffeomorphic contact manifolds $(M_t,\zeta_t)$ for $0 \le t \le 1$ given by
 \[ M_t = \psi_t( \varphi(L(n+1,n))), \quad \zeta_t = TM_t \cap J_{\C^3} (TM_t)\] where 
\[ M_0 = \psi_0(\varphi(L(n+1,n))) = \varphi(L(n+1,n)), \quad M_1 = \psi_1 (\varphi(L(n+1,n))) = L_{A_n}.\] 
\end{rem}

Moreover, we can relate the standard contact structure on $L(n+1,n)$ under the image of $\varphi$.  To avoid excessive parentheses, we use $S^3/A_n$ in place of $L(n+1,n)$ in this lemma.
\begin{lem}\label{technicalphi}
On $\varphi(S^3/A_n), \ \ \varphi_*\xi_{std}= T(\varphi(S^3/A_n)) \cap J_{\C^3} (T(\varphi(S^3/A_n))) .$
\end{lem}

\begin{proof}
Since $A_n \subset SL(2,\C)$ 
we have
\[
\tilde \varphi (J_{\C^2}TS^3) = J_{\C^3}(T\tilde \varphi(S^3)).
\]
We examine $\varphi_*\big(\xi_{\text{std}}\big)$:
\begin{align*}
  \varphi_*(T(S^3/A_n) \cap J_{\C^2} T(S^3/A_n)) &=\tilde \varphi_*(TS^3 \cap J_{\C^2}(TS^3)) \\
&=\tilde \varphi_*(TS^3) \cap \tilde \varphi_*(J_{\C^2}(TS^3)) \\
&=\tilde \varphi_*(TS^3) \cap J_{\C^3}\tilde \varphi_*(TS^3) \\
&= T\tilde\varphi(S^3) \cap J_{\C^3}(T\tilde\varphi(S^3)) \\
&= T(\varphi(S^3/A_n)) \cap J_{\C^3}(T\varphi(S^3/A_n)).
\end{align*}
\end{proof}

 Lemmas \ref{oneparameter} and \ref{technicalphi} in conjunction with Remark \ref{helperrem} and Lemma \ref{graycor} yields the following proposition.

\begin{prop}\label{propatlast}
The image of the lens space $(\varphi(L(n+1,n)),  \varphi_*\xi_{std})$ is contactomorphic to $(L_{A_n}, \xi_{0})$.
\end{prop}

It remains to show that $(\varphi(L(n+1,n)), \varphi_*\xi_{std})$ is contactomorphic to $(L(n+1,n),\xi_{std})$.  To accomplish this, we use Moser's Lemma to prove the following lemma.


\begin{lem}\label{moserlem} The manifolds $(\C^2 \setminus \{ \mathbf{0} \},d\lambda_{\C^2})$ and $(\C^2 \setminus \{\mathbf{0} \}, d\tilde\varphi^*\lambda_{\C^3})$ are contactomorphic.

\end{lem}
\begin{proof}
Consider the family of 2-forms 
\[ \oo_t = (1 - t)\oo_{\C^2} + t\tilde \varphi^*\oo_{\C^3}\]
for $0 \le t \le 1$. Then $\oo_t$ is exact because $Y_0$ and $Y$ are  Liouville vector fields for $\C^2\setminus \mathbf{0}$ equipped with the symplectic forms $\oo_{\C^2}$ and $\oo_{\C^3}$ respectively,  thus $d\lambda_t = \oo_t$ for 
\[\lambda_t = (1 - t)\lambda_{\C^2} + t\tilde\varphi^*(\lambda_{\C^3}).\]
for $0 \le t \le 1$.   We claim for each $t \in [0,1]$, $\lambda_t$ is a family of contact forms.

 We compute
\begin{align*}
\frac{2}{i} \tilde\varphi^*d\lambda_{\C^3} & = d(uv)\wedge d(\overline{uv} ) + d(u^{n+1})\wedge d(\bar u^{n+1}) + d(v^{n+1}) \wedge d(\bar v^{n+1}) \\
& = ((n+1)^2 |u|^{2n} + |v|^2)du\wedge d\bar u + 2\Re (u\bar v dv \wedge d\bar u) + ((n+1)^2|v|^{2n} + |u|^2) dv\wedge d\bar v.
\end{align*}
Since  $\oo_t$ is exact for each $t\in[0,1]$, $d(\oo_t)=0$ for each $t\in[0,1]$.  Moreover, a simple calculation reveals that $\oo_t \wedge \oo_t$ is a volume form on $\C^2$ for each $t\in[0,1]$.   Thus we may conclude that, $\oo_t$ is a symplectic form  for each $t\in[0,1]$.  Applying Moser's argument, Theorem \ref{moser}, yields the desired result.   
\end{proof}

This yields the desired corollary.

\begin{cor}\label{atlast}
The manifolds $(L(n+1,n), \ker \lambda_{\C^2})$ and $(L(n+1,n), \ker \varphi^*\lambda_{\C^3})$ are contactomorphic.
\end{cor}

\begin{proof}
Let $\phi:(\C^2 \setminus \{ \mathbf{0} \},d\lambda_{\C^2})$ and $(\C^2 \setminus \{ \mathbf{0} \}, d\tilde\varphi^*\lambda_{\C^3})$ be the symplectomorphism, which exists by Lemma \ref{moserlem}.  It induces the desired contactomorphism.  On $\C^2 \setminus \{ \mathbf{0} \}$,
\[
\phi^*d(\varphi^*\lambda_{\C^3}) = d\lambda_{\C^2},
\]
thus
\[
d\phi^*(\varphi^*\lambda_{\C^3}) = d\lambda_{\C^2}.
\]
So indeed on $L(n+1,n)$,
\[
\begin{array}{lcl}
\phi_*(\xi_{std}) &=& \phi_*(\ker \lambda_{\C^2}) \\
&=& \ker \varphi_* \lambda_{\C^3} \\
& =& \varphi_* \xi_{std}. \\
\end{array}
\]
\end{proof}
Proposition \ref{propatlast} and Corollary \ref{atlast} complete the proof of Theorem \ref{lenslinkcontacto}. 



\begin{thebibliography}{10}
\bibitem{BO}Bourgeois, F. and Oancea, A; \emph{$S^1$-equivariant symplectic homology and linearized contact homology.}	To appear in IMRN.  arXiv:1212.3731 


\bibitem{BOcorrig} Bourgeois, F and Oancea, A; \emph{Erratum to: An exact sequence for contact and symplectic homology}. Invent. Math. 200 (2015), no. 3, 1065--1076
 
\bibitem{Br}
	Brieskorn, E; \emph{Beispiele zur Differentialtopologie von Singularit\"aten,} {{Invent. Math. 2}}, 1- 14, 1966.


\bibitem{G}
	Geiges, H; \emph{An Introduction to Contact Topology,} Cambridge University Press, 2008.
	
\bibitem{GuSH}
Gutt, J; \emph{The positive equivariant symplectic homology as an invariant for some contact manifolds.} arXiv:1503.01443
	
\bibitem{GuCZ}
Gutt, J; \emph{Generalized Conley-Zehnder index.} Ann. Fac. Sci. Toulouse Math. (6) 23 (2014), no. 4, 907-932. 


	
	\bibitem{HN1}Hutchings, M. and Nelson J; \emph{Cylindrical contact homology for dynamically convex contact forms in three dimensions.}  To appear in J. Symplectic Geom. arxiv:1407.2898
	
\bibitem{HN3}Hutchings, M. and Nelson, J; \emph{Invariance and an integral lift of cylindrical contact homology for dynamically convex contact forms.}  In Preparation.

\bibitem{IM} Ito, Y and Reid, M; \emph{The McKay correspondence for finite subgroups of SL$(3,\C)$}. Higher-dimensional complex varieties (Trento, 1994), 221-240, de Gruyter, Berlin, 1996. 

\bibitem{K}Keating, A; \emph{Homological mirror symmetry for hypersurface cusp singularities.} arxiv:1510.08911

\bibitem{O}
	Kwon, M. and van Koert, O; \emph{Brieskorn Manifolds in Contact Topology,} Bull. Lond. Math. Soc. 48 (2016), no. 2, 173--241.
	
\bibitem{MD} McDuff, D. and Salamon, D; \emph{Introduction to Symplectic Topology,} Oxford University Press, 1998.

\bibitem{McL} McLean, M; \emph{Reeb orbits and the minimal discrepancy of an isolated singularity.} Invent. Math. 204 (2016), no. 2, 505--594.

\bibitem{MR}McLean, M and Ritter, A;  \emph{The cohomological McKay correspondence and symplectic homology}, In preparation.

\bibitem{M} Milnor, J; \emph{Singular points of complex hypersurfaces,} Annals of Mathematics Studies, No. 61 Princeton University Press, 1968

\bibitem{moise}Mo\"ise, E; \emph{Affine structures in 3-manifolds. V. The triangulation theorem and Hauptvermutung.}  Ann. of Math. (2) 56, (1952). 96-114. 
	
	
\bibitem{jo1}Nelson, J; \emph{Automatic transversality in contact homology I: Regularity}, Abh. Math. Semin. Univ. Hambg. 85 (2015), no. 2, 125-179.
	
\bibitem{jocompute} Nelson, J; \emph{Automatic transversality in contact homology II: Invariance and computations.} In preparation.


	
\bibitem{R} Ritter, A; \emph{Deformations of symplectic cohomology and exact Lagrangians in ALE spaces.} 
Geom. Funct. Anal. 20 (2010), no. 3, 779--816. 
	
\bibitem{RS}
	Robbin, J. and Salamon, D; \emph{The Maslov Index for Paths}, {{Topology 32}}, 827-844, 1993.

\bibitem{biased} Seidel, P; \emph{A biased view of symplectic cohomology.} Current developments in mathematics, 2006, pg. 211--253. Int. Press, Somerville, MA.
	
\bibitem{Se} Seidel, P; \emph{Fukaya categories and Picard-Lefschetz theory,} Zurich Lectures in Advanced Mathematics. European
Mathematical Society (EMS), 2008.



\bibitem{U}
	Ustilovsky, I; \emph{Contact Homology and Contact Structures on $S^{4m + 1}$,} Internat. Math. Res. Notices 1999, no. 14, 781--791.
\end{thebibliography}
\end{document}